\documentclass[%
   , colorlinks % also like hyperref with CSC colorscheme (this is the default)
]{mpi2015-cscpreprint}

\usepackage{hyperref}
\usepackage[american]{babel}

\usepackage{graphicx}

\usepackage[T1]{fontenc}
\usepackage{lmodern}

\usepackage{amssymb}
\usepackage{amsthm}
\usepackage{mymacros}

\AtBeginDocument{
  \newtheorem{theorem}{Theorem}[section]
  \newtheorem{lemma}[theorem]{Lemma}
  \newtheorem{definition}[theorem]{Definition}
  
  \newtheorem{corollary}[theorem]{Corollary}
  \newtheorem{remark}[theorem]{Remark}
}

\newcommand{\eg}{e.\,g.,\ }
\newcommand{\ie}{i.\,e.,\ }
\newcommand{\M}{\mathcal{M}}
\newcommand{\T}{\mathcal{T}}
\renewcommand{\S}{\mathcal{S}}

\newcommand{\Hn}{ {\mathbb{H}_n} }   % Hermitian n x n
\newcommand{\XWpd}{ {\mathbb{X}^{\raisebox{0.2em}{{\fontsize{3}{2}\selectfont $>$}}}} }  %Hermitian pos def solns to weak LMI: W(X) \geq 0
\newcommand{\XWpdpd}{ {\mathbb{X}^{\raisebox{0.2em}{{\fontsize{3}{2}\selectfont $\gg$}}}} }  %Hermitian pos def solns to strict LMI: W(X)>0
\begin{document}

%%%%%%%%%%%%%%%%%%%%%%%%%%%%%%%%%%%%%%%%%%%%%%%%%%%%%%%%%%%%%%%%%%%%%%%%%%%%%%%%
% PAPER INFORMATION.                                                           %
%%%%%%%%%%%%%%%%%%%%%%%%%%%%%%%%%%%%%%%%%%%%%%%%%%%%%%%%%%%%%%%%%%%%%%%%%%%%%%%%

\title{Robust forms for passive descriptor systems}

\author[$\ast$]{Peter Benner}
\affil[$\ast$]{Max Planck Institute for Dynamics of Complex Technical Systems, 39106 Magdeburg, Germany. \authorcr%
  \email{benner@mpi-magdeburg.mpg.de}, \orcid{0000-0003-3362-4103}}% chktex 8

\author[$\ast$]{Anshul Prajapati}
\affil[$\ast$]{Max Planck Institute for Dynamics of Complex Technical Systems, 39106 Magdeburg, Germany.\authorcr%
  \email{prajapati@mpi-magdeburg.mpg.de}, \orcid{0000-0002-7641-5304
}}% chktex 8

\author[$\dagger$]{Paul Van Dooren}
\affil[$\dagger$]{Department of Mathematical Engineering, Universit{\'e} catholique de Louvain, Louvain-la-Neuve, Belgium.\authorcr%
  \email{paul.vandooren@uclouvain.be}, \orcid{0000-0002-0115-9932}}% chktex 8

\shorttitle{Robust forms of passive descriptor systems}
\shortauthor{P. Benner, A. Prajapati, P. Van Dooren}
\shortdate{}

\keywords{linear matrix inequality, passivity radius, robustness}

\msc{93D09, 93C05, 49M15, 37J25}

% \author{Peter Benner\thanks{Max Planck Institute for Dynamics of Complex Technical Systems, Magdeburg, Germany}, Anshul Prajapati\thanks{Max Planck Institute for Dynamics of Complex Technical Systems, Magdeburg, Germany}, Paul Van Dooren\thanks{Department of Mathematical Engineering, Universit{\'e} catholique de Louvain, Louvain-la-Neuve, Belgium}
% }

\abstract{We show how to compute the passivity radius of a given generalized state-space model of a rational transfer function that represents a passive system, both in the continuous-time and discrete-time cases. 
We then use this to construct robust descriptor realizations. For the continuous-time case, these realizations have a port-Hamiltonian descriptor form. For the discrete-time case a similar result holds for the so-called normalized realizations. The computation of such robust realizations is linked in both these cases to a particular solution of a linear matrix inequality that characterizes the passivity of the transfer function.}
    
\novelty{This paper extends the standard state-space results of the papers \cite{MehV19} and \cite{MehV20} to also handle the descriptor case.}

\maketitle

\section{Introduction}

We consider realizations of linear dynamical systems with a rational transfer function that is \emph{positive-real or passive}. In particular, we study positive-real transfer functions  which play a fundamental role in systems and control theory: they represent \eg spectral density functions of stochastic processes and arise in spectral factorizations.
Positive-real transfer functions can be described using convex sets, and this property has led to the extensive use of convex optimization techniques in this area \cite{BoyEFB94}.
\emph{Passive} systems and their relationships with \emph{positive-real transfer functions} are well studied, starting with the works  \cite{Kal63,Pop73,Wil72a,Wil72b}, and the topic has recently received considerable attention in the context of \emph{port-Hamiltonian system models} \cite{Sch04,SchJ14}.

In this paper we look at generalized state-space realizations of linear time-invariant and positive-real transfer functions. We characterize the so-called passivity radius of a generalized state-space model as the smallest perturbation of the model parameters that makes the system loose the passivity property. We show that there is a subset of such realizations that achieves near optimal robustness, by maximizing its \emph{passivity radius}. 
We derive these results for both continuous-time and discrete-time models. 

In the continuous-time (CT) case we consider linear time-invariant systems described by the equations
\begin{equation} \label{statespace_c}
 \begin{array}{rcl}  E \dot x(t) & = & Ax(t) + B u(t),\ x(0)=x_0,\\
y(t)&=& Cx(t)+Du(t),
\end{array}
\end{equation}
which can also be described by the associated  transfer function
in the {\em Laplace variable $s$~:}
\begin{equation}\label{tf_c}
\mathcal T(s)= D+C(sE-A)^{-1}B.
\end{equation}
In the discrete-time (DT) case we consider linear time-invariant systems described by the equations
\begin{equation} \label{statespace_d}
 \begin{array}{rcl} E x(k+1) & = & Ax(k) + B u(k),\ x(0)=x_0,\\
y(k) &=& Cx(k)+Du(k),
\end{array}
\end{equation}
which now has an associated transfer function in the \emph{time delay} $z$~: 
\begin{equation*} %\label{tf_d}
\mathcal T(z)= D+C(zE-A)^{-1}B.
\end{equation*}
When discussing results that apply to both cases, we will use the generic variable $\lambda$ instead of
$s$ (reserved for CT systems) and $z$ (reserved for DT systems). The above models will in both cases be denoted by the quintuple
\begin{equation} %\label{eq:model}
    \M:=\{A,B,C,D,E\}.
\end{equation}
Here $u:\mathbb R\to\mathbb{C}^m$,   $x:\mathbb R\to \mathbb{C}^n$,  and  $y:\mathbb R\to\mathbb{C}^m$  are vector-valued functions denoting, respectively, the \emph{input}, \emph{state},
and \emph{output} of the system. Denoting real and complex $n$-vectors ($n\times m$ matrices) by $\mathbb R^n$, $\mathbb C^{n}$ ($\mathbb R^{n \times m}$, $\mathbb{C}^{n \times m}$), respectively, the coefficient matrices satisfy $A, E \in \mathbb{C}^{n \times n}$,   $B\in \mathbb{C}^{n \times m}$, $C\in \mathbb{C}^{m \times n}$, and  $D\in \mathbb{C}^{m \times m}$. When the variables $x$, $u$ and $y$ are real vectors, the system model can be chosen real as well, but the analysis is essentially the same (see e.g. Subsection \ref{sub:real}).

We restrict ourselves to systems which are \emph{minimal}, \ie the pair $(\lambda E-A,B)$ is \emph{controllable} (for all $\lambda \in \mathbb C$, $\rank{[ \lambda E-A \; | \; B]}=n$), and the pair $(\lambda E-A,C)$ is \emph{observable} (\ie $(\lambda E^\mathsf{H}-A^\mathsf{H},C^\mathsf{H})$ is controllable). Here, the Hermitian transpose and the transpose of a vector or matrix $V$ is denoted by
$V^{\mathsf{H}}$ and $V^{\mathsf{T}}$, respectively, and the identity matrix is denoted by $I_n$ or $I$ if the dimension is clear from the context.
Throughout this article we will use the following notation.
We denote the set of Hermitian matrices in $\mathbb{C}^{n \times n}$ by $\Hn$.
Positive definiteness (semi-definiteness) of  $M\in \Hn$ is denoted by $M\succ 0$ ($M\succeq 0$). The smallest and largest eigenvalues of a matrix $M\in\Hn$ will be denoted by $\lambda_{\min}(M)$ and $\lambda_{\max}(M)$, respectively.
The real and imaginary parts of a complex matrix $Z$ are written as $\Re (Z)$ and $\Im (Z)$, respectively, and $\imath$ is the imaginary unit.
We consider functions over $\Hn$, which is a vector space if considered as a \emph{real} subspace of $\mathbb{R}^{n\times n}+\imath \mathbb{R}^{n\times n}$.

The paper is organized as follows. In Section~\ref{sec:PH} and its subsections, we characterize the classes of passive descriptor systems, of port-Hamiltonian (pH) descriptor systems, and of normalized descriptor systems. We then introduce in Section~\ref{sec:passrad} the notion of passivity radius and of $X$-passivity radius of passive systems and
show an important optimality property of normalized systems. In Section~\ref{sec:maxpass} we show how to find the optimal certificate $X$ to maximize the passivity radius of a transfer function model. In the last section, we briefly discuss how to find models for a given transfer function, that is as robust as possible in the sense that its passivity radius is maximal. We then give some concluding remarks in a final Section~\ref{sec:conclude}.

\section{Passive systems} \label{sec:PH}

\subsection{The continuous-time case} \label{sub:CT}

The concepts of \emph{positive-realness} and \emph{passivity} are well studied in the CT case. We briefly recall some important properties following \cite{SchJ14}, and refer to the literature for a more detailed survey. Consider a stable, continuous-time system as in (\ref{statespace_c})  and its transfer function $\mathcal T(s)$ as in \eqref{tf_c}. We define the matrix-valued rational function
\[
\Phi_c(s):= \mathcal T^{\mathsf{H}}(-s) + \mathcal T(s),
\]
which coincides with the Hermitian part of ${\mathcal T}(s)$ on the imaginary axis~:
\[  \Phi_c(\imath\omega):= [\mathcal T(\imath\omega)]^{\mathsf{H}} + \mathcal T(\imath \omega).
\]
\begin{definition}\label{definition:passivity}
A stable transfer function $\mathcal T(s)$ is {\em strictly positive-real} \/if $\Phi_c(\imath \omega)\succ 0 $ for all $ \omega\in \mathbb{R}$ and 
it is \emph{positive-real} if $\Phi_c(\imath \omega)\succeq 0 $ for all $\ \omega\in \mathbb{R}$.\footnote{In the literature, a different definition is used, but for stable and rational continuous-time transfer functions, both definitions are equivalent.}

The transfer function $\mathcal T(s)$ is {\em asymptotically stable} if the eigenvalues of $sE-A$ are in the open left half plane, and it is 
{\em stable} if the eigenvalues of $sE-A$ are in the closed left half plane, with any eigenvalues occurring on the imaginary axis being simple.

The transfer function $\mathcal T(s)$ is {\em strictly passive} \/if  it is strictly positive-real and asymptotically stable and it is 
\emph{passive} if it is positive-real and stable. 
\end{definition}

The rational matrix function $\Phi_c(s)$ is the Schur complement of the so-called continuous-time {\em system pencil}
\begin{equation*} %\label{pencil_Lc}
L_c(s) :=
\left[ \begin{array}{cc|c} 0 & A-sE & B \\
A^{\mathsf{H}}+sE^\mathsf{H} & 0 & C^{\mathsf{H}} \\ \hline B^{\mathsf{H}} & C & D^{\mathsf{H}}+D  \end{array} \right],
\end{equation*}
and if the model $\M$ is minimal, then the finite generalized eigenvalues of $L_c(s)$ are the finite zeros of $\Phi_c(s)$.
The following transformation using an arbitrary matrix $X \in \Hn$, leaves the Schur complement, and hence also the rational matrix function $\Phi_c(s)$, unchanged
\begin{equation} \label{Ss}
\left[ \begin{array}{cc|c} 0 & A-sE & B \\
A^{\mathsf{H}}+sE^\mathsf{H} & -A^\mathsf{H}XE - E^{\mathsf{H}}XA & C^{\mathsf{H}} - E^{\mathsf{H}}XB \\ \hline B^{\mathsf{H}} &
C- B^{\mathsf{H}}XE & D^{\mathsf{H}}+D \end{array} \right] \end{equation}
$$ =
\left[ \begin{array}{cc|c} I_n & 0 & 0 \\
-E^\mathsf{H}X & I_n & 0 \\ \hline 0 & 0 & I_m  \end{array} \right] L_c(s)
\left[ \begin{array}{cc|c} I_n & -XE & 0 \\
0 & I_n & 0 \\ \hline 0 & 0 & I_m  \end{array} \right].
$$
Let us now denote by $W_c(X,\M)$ the following submatrix of \eqref{Ss}
\begin{equation} \label{KYP-LMI_c}
W_c(X,\M) := \left[
\begin{array}{cc}
-A^\mathsf{H}XE - E^{\mathsf{H}}XA & C^{\mathsf{H}} - E^{\mathsf{H}}XB \\
C- B^{\mathsf{H}}XE & D^{\mathsf{H}}+D 
\end{array}
\right].
\end{equation}
Then, it can be shown that
\[
\Phi_c(s) =
\left[ \begin{array}{cc} B^{\mathsf{H}}(A^\mathsf{H} + sE^{\mathsf{H}})^{-1} & -I_m  \end{array} \right]
\, W_c(X,\M) \left[ \begin{array}{c} (A-sE)^{-1}B \\ -I_m \end{array} \right], 
\]
and that $\mathcal T(s)$ is positive-real and stable if and only if there exists a positive definite matrix $X\in \Hn$ such that the following LMI holds 
\begin{equation} \label{LMI_c}
W_c(X,\M) \succeq 0.
\end{equation}
Such a matrix $X$ is also called a {\em certificate} of the LMI. 
We point out that when $E$ is singular, $\Phi_c(\infty)$ is singular as well and $\T(s)$ can not be strictly passive. Hence strict passivity implies that $E$ is invertible.

In this context we will make use of the sets
\begin{subequations} \label{LMIsolnsets}
\begin{align}
&\XWpd :=\left\{ X\in \Hn \left|   W(X) \succeq 0,\ X \succ 0 \right.\right\}, %\label{XpdsolnWpsd} 
\\[1mm]
&\XWpdpd :=\left\{ X\in \Hn \left|   W(X) \succ 0,\ X \succ 0 \right.\right\}, %\label{XpdsolWpd}
\end{align}
\end{subequations}
where the LMI matrix $W(X)$ stands for both $W_c(X,\M)$ (in the CT case)
and $W_d(X,\M)$ (in the DT case, see Section\ref{sub:DT}).
A CT system $\M :=\left\{A,B,C,D,E\right\}$ is  \emph{passive} if there exists a state-dependent
\emph{storage function}, $\mathcal H(x) \geq 0$, such that for any $t_1,t_0\in \mathbb R$ with $t_1>t_0$,
 the \emph{dissipation inequality}
\begin{equation} \label{supply} \mathcal H(x(t_1))-\mathcal H(x(t_0)) \le \int_{t_0}^{t_1} \Re (y(t)^{\mathsf{H}}u(t)) \, dt
\end{equation}
holds.
If for all $t_1>t_0$, the inequality in \eqref{supply}
is strict then the system is called \emph{strictly passive}. Note that the above dissipation inequality is equivalent to (strict) passivity of the associated transfer function, as defined in Definition~\eqref{definition:passivity}.

If $D^{\mathsf{H}}+D$ is invertible, then
the minimum rank solutions of \eqref{LMI_c} in $\XWpd$
are those for which $\rank{W_c(X,\M)} = \rank{D^{\mathsf{H}}+D}  = m$, which in turn is the case
if and only if the Schur complement of $D^{\mathsf{H}}+D$ in $W_c(X,\M)$ is zero.  This Schur
complement is associated with the continuous-time \emph{algebraic Riccati equation (CARE)}
\begin{equation} \label{riccatic}
- A^{\mathsf{H}}XE -E^{\mathsf{H}}XA -(C^{\mathsf{H}}-E^\mathsf{H}XB)(D^{\mathsf{H}}+D)^{-1}(C-B^{\mathsf{H}}XE)=0.
\end{equation}
Solutions $X$ to (\ref{riccatic}) yield a spectral factorization of $\Phi_c(s)$, and each solution corresponds to an invariant subspace 
of the so-called \emph{Hamiltonian matrix} 
\begin{equation*}%\label{HamMatrix}
H:=\left[\begin{array}{cc} AE^{-1}-B (D^{\mathsf{H}}+D)^{-1} CE^{-1} & - B (D^{\mathsf{H}}+D)^{-1} B^{\mathsf{H}} \\
E^\mathsf{{-H}}C^{\mathsf{H}} (D^{\mathsf{H}}+D)^{-1} CE^{-1} & -(AE^{-1}-B (D^{\mathsf{H}}+D)^{-1} CE^{-1})^{\mathsf{H}} \end{array} \right].
\end{equation*}
Notice that we require $E$ to be invertible in order to write the Hamiltonian matrix $H$ in that form.
It is shown in \cite{Wil71} that for a minimal system $\M$, the set of solutions $X$ of the Riccati equation \eqref{riccatic} has two extremal solutions
$X_+$ and $X_-$ such that  all other solutions $X$ satisfy 
\begin{equation} \label{eq:Ricbound}
     \infty I \succ X_+ \succeq X \succeq X_- \succ 0 
\end{equation}
Notice that this implies that the set $\XWpd$ is a bounded convex set.

\subsection{Port-Hamiltonian systems} \label{sub:PH}
A special class of realizations of CT passive systems is that of \emph{port-Hamiltonian systems}.
\begin{definition}\label{def:ph}
A linear time-invariant \emph{port-Hamiltonian (pH) descriptor system} has the generalized state-space form
\begin{equation} \label{eq:pH}
 \begin{array}{rcl} E \dot x(t)  & = & (J-R)Q x(t) + (G-K) u(t),\\
y(t) &=& (G+K)^{\mathsf{H}}Q x(t)+(S+N)u(t),
\end{array}
\end{equation}
and the system model matrices satisfy the structural conditions
\[
\mathcal V:= \left[ \begin{array}{cccc} -J & -G \\ \!G^{\mathsf{H}} \! & N \end{array} \right]=-\mathcal V^{\mathsf{H}},\
\mathcal W:= \!\left[ \begin{array}{cccc} R & \!K \! \\ K^{\mathsf{H}} & S \end{array} \right] \!=\mathcal W^{\mathsf{H}}\succeq 0, \ E^{\mathsf{H}}Q=Q^{\mathsf{H}}E \succeq 0.
\]
\end{definition}
Port-Hamiltonian descriptor systems were introduced from a different point of view in \cite{BeattieMXZ}, but they also have a storage function $\mathcal{H}(x):=\frac12(x^\mathsf{H}Q^\mathsf{H}Ex)$ and satisfy a dissipation inequality, and hence they are passive. Thus,  there must be a coordinate transformation between a passive system and a representation ({\ref{eq:pH}) as a pH descriptor system. We give here a  construction of such a possible transformation.

Consider a minimal descriptor model $\M:=\{A,B,C,D,E\}$ of a passive linear time-invariant system and choose $Q:=c^2 E, \; c>0$. The condition $E^\mathsf{H}Q=Q^\mathsf{H}E\succeq 0$ is then satisfied and the storage function equals $\mathcal{H}(x):=\frac12(\|cEx\|_2^2)$.
One can then always apply a generalized state-space transformation as follows
\begin{equation} \label{gsst}
 \{\hat A,\hat B,\hat C,\hat D,\hat E\}:= \{TAS, TB, CS, D,TES \},  \quad \det S\neq 0, \;\; \det T \neq 0,
\end{equation}
and choose $S$ and/or $T$, without affecting the transfer function of the model. 
Let us now look at the remaining parameters of the model. Since the transfer function is passive, there is a solution $X\in \XWpd$ of the LMI \eqref{LMI_c}.
We then use a symmetric factorization $X= c^2T^{\mathsf{H}}T$, where $c>0$, and define a new realization
\begin{equation*}
    %\label{gsstT}
\M_T := \{A_T, B_T, C_T, D_T,E_T \}:= \{TAT^\mathsf{H}, TB, CT^\mathsf{H}, D,TET^\mathsf{H} \}.
\end{equation*}
Moreover, the LMI \eqref{LMI_c} transforms to 
\[\
\left[ \begin{array}{cccc} T & 0\\ 0 & I_m
\end{array}
\right]
\left[ \begin{array}{cccc} -A^{\mathsf{H}}XE-E^{\mathsf{H}}XA & C^{\mathsf{H}}-E^{\mathsf{H}}XB \\ C-B^{\mathsf{H}}XE & D^{\mathsf{H}}+D
\end{array}
\right]
\left[ \begin{array}{cccc} T^{\mathsf{H}} & 0\\ 0 & I_m
\end{array}
\right] \succeq 0 ,
\]
which can be rewritten as
\begin{equation} \label{gssPH}
 W_c(c^2I_n,\M_T) := \left[ \begin{array}{cc} c^2E_T^{\mathsf{H}} &  \\ 
 & I_m \end{array} \right]
\left[ \begin{array}{cccc}-A_T & -B_T \\ C_T & D_T
\end{array}
\right]    
+
\left[ \begin{array}{cccc} -A_T^\mathsf{H} & C_T^\mathsf{H} \\ - B_T^\mathsf{H} & D_T^\mathsf{H}
\end{array}
\right]  \left[ \begin{array}{cccc} c^2E_T &  \\ & I_m
\end{array} 
\right] 
\succeq 0. 
\end{equation}
We can then use the Hermitian and skew-Hermitian part of the matrix
\[
 \S := \left[ \begin{array}{cccc}-A_T & -B_T \\ C_T & D_T
\end{array}
\right]  \left[ \begin{array}{cc} c^2E_T &  \\ & I_m
\end{array} \right]^{-1}
= \left[ \begin{array}{cccc}-A_T & -B_T \\ C_T & D_T
\end{array}
\right]  \left[ \begin{array}{cc} Q_T &  \\ & I_m
\end{array} \right]^{-1}
\] 
to define the coefficients of a pH representation via
\[
 \left[ \begin{array}{cccc} R & K \\ K^{\mathsf{H}} & S
\end{array}
\right] := \frac{{\S} +{\S}^{\mathsf{H}}}{2} \succeq 0, \quad \left[ \begin{array}{cccc} -J &  -G \\ G^{\mathsf{H}} & N
\end{array}
\right] :=  \frac{{\S} - {\S}^{\mathsf{H}}}{2},
\]
and which corresponds to the descriptor pH model \eqref{eq:pH}
with $Q_T = c^2 E_T$. The generalized state-space equivalence transformation that was used for this is of the form 
\eqref{gsst} and it preserves the positive semi-definiteness of the storage function. Note that meanwhile, we constructed a coordinate system in which the LMI is satisfied with a ``normalized" certificate $X=c^2I_n$.
Also, the factor $T$ is unique up to a unitary factor $U$, since $T^{\mathsf{H}}U^{\mathsf{H}}UT=T^{\mathsf{H}}T$, but this factor $U$ will not affect the results described in this paper.

\subsection{The DT case} \label{sub:DT}

The concept of \emph{passivity} for discrete-time systems is described in  \cite{Wil72b}, \cite{MehV20}, and we refer to those papers for proofs and for a more detailed survey. Consider a discrete-time system (\ref{statespace_d}) with minimal generalized state-space model $$\M:=\{A,B,C,D,E\}$$ 
and transfer function $\mathcal T(z):=C(zE-A)^{-1}B+D$ and define the rational matrix function of $z\in \mathbb{C}$~: 
\[ \Phi_d(z):=\mathcal T^{\mathsf{H}}(z^{-1}) + \mathcal T(z), \] which coincides with the Hermitian part of $\mathcal T(z)$ on the  unit circle:
\[ \Phi_d(e^{\imath \omega})=[\mathcal T(e^{\imath \omega})]^\mathsf{H} + \mathcal T(e^{\imath \omega}). \] 

\begin{definition}
A stable transfer function $\mathcal T(z)$ is {\em strictly positive-real} \/if $\Phi_d(e^{\imath \omega})\succ 0 $ for all $\ \omega\in [-\pi,\pi]$ and 
it is \emph{positive-real} if $\Phi_d(e^{\imath \omega})\succeq 0 $ for all $\ \omega\in [-\pi,\pi]$.\footnote{In the literature, a different definition is used, but for stable and rational discrete-time transfer functions, both definitions are equivalent.}

The transfer function $\mathcal T(z)$ is {\em asymptotically stable} if the eigenvalues of $zE-A$ are in the open unit disc, and it is 
{\em stable} if the eigenvalues of $zE-A$ are in the closed unit disc, with any eigenvalues occurring on the unit circle being simple.

The transfer function $\mathcal T(z)$ is {\em strictly passive} \/if  it is strictly positive-real and asymptotically stable and it is 
\emph{passive} if it is positive-real and stable. 
\end{definition}

The rational matrix function $\Phi_d(z)$ is the Schur complement of the so-called discrete-time {\em system pencil}
\begin{equation*} %\label{pencil_Lz}
L_d(z) :=
\left[ \begin{array}{cc|c} 0 & A-zE & B \\
zA^{\mathsf{H}}-E^\mathsf{H} & 0 & C^{\mathsf{H}} \\ \hline zB^{\mathsf{H}} & C & D^{\mathsf{H}}+D  \end{array} \right]
\end{equation*}
and if the model $\M$ is minimal, then the finite generalized eigenvalues of $L_d(z)$ are the finite zeros of $\Phi_d(z)$.
The following transformation using an arbitrary matrix $X \in \Hn$, leaves the Schur complement, and hence also the rational matrix function $\Phi_d(z)$, unchanged
\begin{equation} \label{Sz}
\left[ \begin{array}{cc|c} 0 & A-zE & B \\
zA^{\mathsf{H}}-E^\mathsf{H} & E^\mathsf{H}XE - A^{\mathsf{H}}XA & C^{\mathsf{H}} - A^{\mathsf{H}}XB \\ \hline zB^{\mathsf{H}} &
C- B^{\mathsf{H}}XA & D^{\mathsf{H}}+D -B^{\mathsf{H}}XB \end{array} \right] \end{equation}
$$ =
\left[ \begin{array}{cc|c} I_n & 0 & 0 \\
-A^\mathsf{H}X & I_n & 0 \\ \hline -B^\mathsf{H}X & 0 & I_m  \end{array} \right] L_d(z)
\left[ \begin{array}{cc|c} I_n & -XE & 0 \\
0 & I_n & 0 \\ \hline 0 & 0 & I_m  \end{array} \right].
$$
Let us now denote by $W_d(X,\M)$ the following submatrix of \eqref{Sz}
\begin{equation} \label{KYP-LMI_d}
W_d(X,\M) := \left[
\begin{array}{cc}
E^\mathsf{H}XE - A^{\mathsf{H}}XA & C^{\mathsf{H}} - A^{\mathsf{H}}XB \\
C- B^{\mathsf{H}}XA & D^{\mathsf{H}}+D -B^{\mathsf{H}}XB
\end{array}
\right].
\end{equation}
Then it can be shown that
\[
\Phi_d(z) =
\left[ \begin{array}{cc} B^{\mathsf{H}}(A^{\mathsf{H}}-z^{-1}\,E^\mathsf{H} )^{-1} & -I_m  \end{array} \right]
\, W_d(X,\M) \left[ \begin{array}{c} (A-z\,E)^{-1}B \\ -I_m \end{array} \right], 
\]
and that $\mathcal T(z)$ is positive-real and stable if and only if there exists a positive definite matrix $X\in \Hn$, called a {\em certificate}, such that the following LMI holds 
\begin{equation} \label{LMI_d}
W_d(X,\M) \succeq 0.
\end{equation} 

An important subset of ${\XWpd}$ are those certificates for which the
rank $r$ of $W_d(X,\M)$ is minimal ({\ie} for which $r=\rank{\Phi_d(z)}$).
If $D^{\mathsf{H}}+D-B^{\mathsf{H}}XB$ is invertible, then
the minimum rank solutions in $\XWpd$
are those for which $\rank{W_d(X,\M)} = \rank{D^{\mathsf{H}}+D-B^{\mathsf{H}}XB}  = m$, which in turn is the case
if and only if the Schur complement of $D^{\mathsf{H}}+D-B^{\mathsf{H}}XB$ in $W_d(X,\M)$ is zero.  This Schur complement is associated with the discrete-time \emph{algebraic Riccati equation (DARE)}
\begin{equation}
 E^\mathsf{H}XE -A^{\mathsf{H}}XA - (C^{\mathsf{H}}-A^{\mathsf{H}}XB)(D^{\mathsf{H}}+D-B^{\mathsf{H}}XB)^{-1}(C-B^{\mathsf{H}}XA)=0.\label{riccatid}
\end{equation}
Solutions $X$ to (\ref{riccatid}) produce a spectral factorization of $\Phi_d(z)$, and each solution corresponds to an
\emph{invariant subspace} 
that remains invariant under the multiplication with the matrix
\begin{equation*} %\label{SymMatrix}
S :=\left[\begin{array}{cc} E &  B (D^{\mathsf{H}}+D)^{-1} B^{\mathsf{H}} \\
 0 & (A-B (D^{\mathsf{H}}+D)^{-1} C)^{\mathsf{H}} \end{array}\right]^{-1} \left[\begin{array}{cc} A-B (D^{\mathsf{H}}+D)^{-1} C & 0 \\
C^{\mathsf{H}} (D^{\mathsf{H}}+D)^{-1} C & E^\mathsf{H} \end{array}\right].
\end{equation*}
Such a subspace is called a \emph{Lagrangian invariant subspace} and the 
matrix $S$ has a \emph{symplectic structure} (see e.g., \cite{Meh91}).
Each solution $X$ of \eqref{riccatid} can also be associated with an \emph{extended Lagrangian invariant subspace}
for the pencil $L_d(z)$, spanned by the columns of
$ \widehat{U}:=\left[\begin{array}{ccc} -X^{\mathsf{T}}
& I_n & -F^{\mathsf{T}} \end{array}\right]^{\mathsf{T}}$.
 In particular, $\widehat{U}$ satisfies
\[
\left[ \begin{array}{ccc} 0 & A & B \\
	-E^\mathsf{H} & 0 & C^{\mathsf{H}} \\  0 & C & D^{\mathsf{H}}+D  \end{array} \right] \widehat{U}
  =\left[ \begin{array}{ccc} 0 & E & 0\\
	-A^{\mathsf{H}} & 0 & 0\\ -B^{\mathsf{H}} & 0 & 0 \end{array} \right] \widehat{U} A_{F}.
\]
If $D^{\mathsf{H}}+D-B^{\mathsf{H}}XB$ is not invertible
then more complicated constructions are necessary, see \cite{Meh91}.

In the continuous-time case, the definition of a passive systems has its origin in network theory, but its formal definition is 
associated with the existence of a storage function and a particular dissipation inequality.
The equivalent concept for the discrete-time case again follows from the LMI \eqref{LMI_d}.
If we define the vector $z_k$ as the stacked vector of the state $x_k$ above the input $u_k$, and construct the inner product
$z_k^{\mathsf{H}} W_d(X,\M) z_k$, then we obtain the inequality
\begin{equation*} %\label{dissipation}
 \frac12 \left(x_k^\mathsf{H}E^\mathsf{H}Ex_k - x_{k+1}^\mathsf{H}E^\mathsf{H}Ex_{k+1}\right) + y_k^{\mathsf{H}}u_k + u_k^{\mathsf{H}}y_k =  z_k^{\mathsf{H}}W_d(X) z_k \ge 0.
\end{equation*}
Using the quadratic storage function $\mathcal H(x_i):=\frac12 \|Ex_i\|^2_2$ this yields a dissipation inequality 
$$  \mathcal H(x_{k})- \mathcal H(x_{0}) \le \sum_{i=0}^{k-1} \Re(y_i^{\mathsf{H}}u_i) 
$$
that is similar to the one of the continuous-time formulation. 
It follows from the continuous-time literature \cite{Wil72b} and the bilinear transformation between continuous-time and discrete-time systems  that if the system $\M$ of (\ref{statespace_d}) is minimal, then the LMI \eqref{LMI_d}
has a solution $X \succ 0$ if and only if $\M$ is a passive system. Moreover, the solutions of \eqref{LMI_d} also satisfy the matrix inequalities
\eqref{eq:Ricbound}, and therefore they form a bounded set, which we call $\mathbb{X}^\pm$. We thus have the following inclusions
$$ \XWpdpd  \subset \XWpd  \subset \mathbb{X}^\pm .
$$ 
Notice also that the $(1,1)$ block in the LMI \eqref{LMI_d} 
is a discrete-time Lyapunov equation with $X\succ 0$. This implies that $(E,A)$ is asymptotically stable if $W_d(X,\M) \succ 0$ and is stable if $W_d(X,M)\succeq 0$, see also \cite{LanT85}. 
It is also known that if the system is strictly passive, meaning that $\Phi_d(e^{\imath \omega})\succ 0$ for the whole unit circle, then $X_+\succ X_-$.

The bilinear transformation between continuous-time and discrete-time systems preserves the solution sets $\XWpdpd$ and $\XWpd$
as well as the solutions $X_+$ and $X_-$ of the Riccati equation. It was shown \eg in \cite{MehV19} that the set $\mathbb{X}^\pm$ has a nonempty interior if and only if $X_+ \succ X_-$. Since  $\XWpd$ is a subset of $\mathbb{X}^\pm$ it also follows $\XWpd$ has an empty interior when $X_+-X_-$ is singular. 
The proof for the discrete-time case again follows from the bilinear transformation.

\subsection{Normalized passive realizations} \label{sec:norm}
For the case of discrete-time passive systems, the class of realizations that have special properties, are the ones associated with a normalized certificate $X=c^2I_n$, where $c>0$. 

\begin{definition} \label{pH}
A {\em normalized passive system} has the generalized state-space form \eqref{statespace_d}
where the system matrices satisfy the matrix inequality
\begin{equation*} 
W_d(c^2I_n,\M_T):=\left[\begin{array}{cccc} c^2E_T^\mathsf{H}E_T & C_T^{\mathsf{H}} \\ C_T & D_T^{\mathsf{H}}+D_T\end{array}\right]-
c^2\left[\begin{array}{cc} A_T^{\mathsf{H}} \\ B_T^{\mathsf{H}} \end{array}\right]\left[\begin{array}{cc} A_T & B_T \end{array}\right]
  \succeq 0, \quad c>0.
\end{equation*}
\end{definition}
We now show that every passive system has an equivalent normalized passive realization.
Consider a minimal state-space model  $\M:=\{A,B,C,D,E\}$ of a passive linear time-invariant system and let  $X\in \XWpd$ be a solution of the LMI \eqref{LMI_d}. 
We then use the symmetric factorization $X= c^2 T^{\mathsf{H}}T$ for some $c>0$, and define a new generalized state-space realization
\[
\M_T:=\{A_T,B_T,C_T,D_T,E_T\} := \{TAT^\mathsf{H}, TB, CT^\mathsf{H}, D, TET^\mathsf{H}\}
\]
so that 
\begin{eqnarray*}  
\left[ \begin{array}{cccc} T & 0\\ 0 & I_m
\end{array}
\right]
\left[ \begin{array}{cccc} E^{\mathsf{H}}XE-A^{\mathsf{H}}XA & C^{\mathsf{H}}-A^{\mathsf{H}}XB \\ C-B^{\mathsf{H}}XA & D^{\mathsf{H}}+D-B^{\mathsf{H}}XB
\end{array}
\right]
\left[ \begin{array}{cccc} T^\mathsf{H} & 0\\ 0 & I_m
\end{array}
\right] \\
= 
\left[\begin{array}{cccc} c^2 E_T^{\mathsf{H}}E_T & C_T^{\mathsf{H}} \\ C_T & D_T^{\mathsf{H}}+D_T\end{array}\right]- c^2
\left[\begin{array}{cc} A_T^{\mathsf{H}} \\ B_T^{\mathsf{H}} \end{array}\right]\left[\begin{array}{cc} A_T & B_T \end{array}\right]
\succeq 0,
\end{eqnarray*}
which expresses that the transformed realization $\M_T$ is now normalized. 
Notice that the factor $T$ is unique up to a unitary factor $U$ since $T^{\mathsf{H}}U^{\mathsf{H}}UT=T^{\mathsf{H}}T$.
This unitary factor does not affect the normalization constraint, but we could choose it to put $A_T$ in a special coordinate system, like the upper triangular Schur form. 

Even after the normalization,  there is still a lot more freedom in the representation of the system, since we could have used any matrix $X$ from the set $\XWpd$ to normalize our realization.
In the remainder of this paper, we will focus on normalized passive realizations. The freedom remaining is thus the choice of the matrix $X$ from $\XWpd$, which, as we will see, can be used to make the representation more robust or less sensitive to perturbations. The remainder of this paper will deal with the question of how to make use of this freedom in the generalized state-space transformation to determine an improved normalized realization.

\subsection{Choices of representation}  \label{sub:real}

An important degree of freedom for descriptor systems is that the transfer function $\T(\lambda)$ does not change when scaling the system as follows \[ \hat\M:= \{\hat A,\hat B, \hat C, \hat D,\hat E\}:=
 \{c_1c_2A,c_1B,c_2C,D,c_1c_2E\} \quad c_1,c_2>0,  \]
since this is a special case of the general state-space transformation \eqref{gsst}. Exploiting such a degree of freedom has an important implication as explained in the following Lemma, which applies to both the CT and DT cases.
\begin{lemma}
Let the system model $\M:=\{A,B,C,D,E\}$ be normalized with certificate
$X:=c^2 I_n, \; c>0$. Then the scaled model 
    $$\M_c:=\{A_c,B_c,C_c,D_c,E_c\}:=\{cA,cB,C,D,cE\}$$
    is normalized with certificate $X=I_n$
\end{lemma}
\begin{proof}
For the CT case this follows from
$$ W(c^2 I_n,\M)=\left[\begin{array}{cc}
 -c^2 E^\mathsf{H} & 0 \\ 0 & I_m 
\end{array} \right] \left[\begin{array}{cc}
 A & B \\ C & D \end{array} \right] + \left[\begin{array}{cc}
 A & B \\ C & D \end{array} \right] ^\mathsf{H} \left[\begin{array}{cc}
 -c^2 E & 0 \\ 0 & I_m 
\end{array} \right] \succeq 0, $$
which can be rewritten as
$$ W(I_n,\M_c)=\left[\begin{array}{cc}
 -c E^\mathsf{H} & 0 \\ 0 & I_m 
\end{array} \right] \left[\begin{array}{cc}
 cA & cB \\ C & D \end{array} \right] + \left[\begin{array}{cc}
 cA & cB \\ C & D \end{array} \right] ^\mathsf{H} \left[\begin{array}{cc}
 -c E & 0 \\ 0 & I_m 
\end{array} \right] \succeq 0. $$
For the DT case this follows from
$$ W(c^2 I_n,\M)=\left[\begin{array}{cc}
 -c^2 E^\mathsf{H}E & C^\mathsf{H} \\ C & D^\mathsf{H}+D 
\end{array} \right] - c^2\left[\begin{array}{cc}
 A^\mathsf{H} \\ B^\mathsf{H} \end{array} \right]\left[\begin{array}{cc}
 A & B \end{array} \right]  \succeq 0, $$
which can be rewritten as
$$ W(I_n,\M_c)=\left[\begin{array}{cc}
 -(cE^\mathsf{H})(cE) & C^\mathsf{H} \\ C & D^\mathsf{H}+D 
\end{array} \right] - \left[\begin{array}{cc}
 cA^\mathsf{H} \\ cB^\mathsf{H} \end{array} \right]\left[\begin{array}{cc}
 cA & cB \end{array} \right]  \succeq 0. $$
\end{proof}

This property also raises the question about bounding the contant $c$, both from below and above, since otherwise, we would obtain a disproportionate
quintuple $\{A,B,C,D,E\}$. Choosing $c^2=\det X$, \eg yields a transformation $T$ with $\det T=1$ and therefore preserves the determinant of the matrices $E$ and $A$ in our model. 
Choosing $c^2=\lambda_{\min}(X)$, yields $\sigma_{\min}(T)=1$, an option that we will use for optimizing the passivity radius.

When $E$ is assumed invertible, we could also have substituted $\hat x(t)= E x(t)$ and $Q=E$ in the pH representation \eqref{eq:pH} which would reduce the problem to one involving a standard state-space model (see \cite{MehV19}). There is even more freedom in the representation of the system, since we could have started from any matrix $X$ from the set $\XWpd$, to normalize the model, or in pH models we could have chosen a different matrix $Q$.

\medskip

We stress that when the model $\M$ is real, then all the definitions and properties discussed above still hold. Moreover, the sets $\XWpd$ and $\XWpdpd$ can be constrained to be
real without altering any of the results, since the real part $X_\Re$ of
a Hermitian matrix $X$ is symmetric and positive (semi-)definite whenever $X$ is Hermitian and (semi-)definite. When the model $\M$ is real it therefore follows that whenever $W(X)\succeq 0$ (or $W(X) \succ 0$) then we also have that $W(X_\Re)\succeq 0$ (or $W(X_\Re)\succ 0$), and it then suffices to verify these conditions over the real symmetric matrices only. When doing that, the corresponding pH realizations will also be real. Finally, we point out that the extremal solutions $X_+$ and $X_-$ of the Riccati equations are also real when the model $\M$ is real.

\section{The passivity radius}\label{sec:passrad}

In order to analyze the robustness of representations of a passive system we need to introduce a robustness measure, which we will call the \emph{passivity radius} $\rho_{\M}$ of the model $\M$~: it is the smallest perturbation (in an appropriate norm) to the coefficients of a model $\M $ that makes the system fail to be {\em strictly passive} (see also \cite{BeaMV19,MehV19,MehV20} where such a measure was already introduced).
For every certificate $X\in \XWpdpd$ of $W(X,\M) \succ 0 $ (meaning the LMI~\eqref{LMI_c} in the continuous-time case, and the LMI~\eqref{LMI_d} in the discrete-time case), we can look for the smallest perturbation $\Delta_\M$ to our model $\M $ that makes $\det W(X,\M+\Delta_\M)=0$, since this is what happens when  $W(X,\M+\Delta_\M)$ is not positive definite anymore. But the parameters of $\Delta_\M$ do not enter linearly in 
$W(X,\M+\Delta_\M)$ and the use of LMI techniques may then fail.
Therefore we will not perturb $E$ and only perturb the other coefficients.
We thus look for the smallest perturbation 
\begin{equation*}  %\label{eq:delta}
    \Delta_\M :=\{ \Delta_A,\Delta_B,\Delta_C,\Delta_D,0\},
\end{equation*} 
to our model $\M $ that makes $\det W(X,\M+\Delta_\M)=0$. 
To measure the size of the perturbation $\Delta_\M$ of a generalized state space model $\M $ we will use the Frobenius norm or the spectral norm
\begin{equation} \label{Delta}
\Delta := \left[\begin{array}{ccc}
\Delta_A & \Delta_B \\
\Delta_C & \Delta_D
\end{array}\right] , \quad 
 \|\Delta_\M \|_F := \left \| \Delta \right \|_F , \quad
 \|\Delta_\M \|_2 := \left \|
\Delta \right \|_2.
\end{equation}
We need to treat the CT and DT cases differently because the parameters of the model $\M$ need to appear linearly to allow the use of LMI techniques.

\subsection{The CT case}
As soon as we have identified a valid certificate $X$, we can consider it fixed, and when keeping $E$ fixed as well, the matrix \eqref{KYP-LMI_c} becomes linear in the remaining parameters $\{A,B,C,D\}$. Its perturbed
version then yields an LMI in the parameters of $\Delta_\M$~:
\[
 W_c(X,\M+\Delta_\M) = \left[
\begin{array}{cc}
 0 & (C+\Delta_C)^{\mathsf{H}} \\
(C+\Delta_C) & (D+\Delta_D)^{\mathsf{H}}+(D+\Delta_D)
\end{array}
\right]\]
\begin{equation} \label{perturb_c} - \left[\begin{array}{cc}
(A+\Delta_A)^{\mathsf{H}}XE + E^{\mathsf{H}}X(A+\Delta_A) & E^{\mathsf{H}}X\,(B+\Delta_B) \\
(B+\Delta_B)^{\mathsf{H}}XE & 0
\end{array}
\right]\succ 0.
\end{equation}
Notice that this LMI can be rewritten in a blocked form
\begin{equation}
\left[\begin{array}{cc} -E^{\mathsf{H}}X & 0 \\ 0 & I_m \end{array}\right]
\left[\begin{array}{cc} A+\Delta_A & B+\Delta_B \\ C +\Delta_C & D+\Delta_D \end{array}\right]
+ 
\left[\begin{array}{cc} A^{\mathsf{H}}+\Delta_A^{\mathsf{H}} & C^{\mathsf{H}}+\Delta_C^{\mathsf{H}} \\ B^{\mathsf{H}}+\Delta_B^{\mathsf{H}} & D^{\mathsf{H}}+\Delta_D^{\mathsf{H}} \end{array}\right]
 \left[\begin{array}{cc}-XE & 0 \\ 0 & I_m \end{array}\right]
 \succ 0, \label{hatXW}
\end{equation}
and by setting
\begin{equation} \label{defWX}
W_0:=W_c(X,\M), \quad F := \left[\begin{array}{cc} -XE & 0 \\ 0 & I_m \end{array}\right],\quad G := \left[\begin{array}{cc} I_n & 0 \\ 0 & I_m \end{array}\right], \quad \Delta := \left[\begin{array}{cc} \Delta_A & \Delta_B \\ \Delta_C & \Delta_D \end{array}\right],
\end{equation}
inequality \eqref{perturb_c} can be written in the more compact form
\begin{equation} \label{WDelta}
W_c(X,\M+\Delta_\M)=  W_0 +  \left[\begin{array}{cc} F^\mathsf{H} & G^\mathsf{H} \end{array}\right] \left[\begin{array}{cc} 0 & \Delta \\ \Delta^\mathsf{H} & 0 
\end{array}\right] \left[\begin{array}{c} F \\ G \end{array}\right] \succ 0 .
\end{equation}
This LMI has to hold as long as the system is still strictly passive. In order to violate this condition, we need to find the smallest $\Delta$ such that $\det W_c(X,\M+\Delta_\M) =0$. 

\subsection{The DT case} 

Even when $E$ is kept fixed, 
the matrix \eqref{KYP-LMI_d} appearing in the LMI condition for the DT case is
not linear in the unknown matrices $\{A,B,C,D\}$, but on the other hand, it is the Schur complement of the following (larger) LMI that is linear in these parameters~:
\begin{equation*} %\label{eq:largeLMI}
\widehat W_d(X,\M):=\left[\begin{array}{ccc} X &  XA & XB \\  A^{\mathsf{H}}X & E^{\mathsf{H}}XE & C^{\mathsf{H}} \\ B^{\mathsf{H}}X &
C & D^{\mathsf{H}}+ D  \end{array}\right] .
\end{equation*}
We point out that both LMI's $\widehat W_d(X,\M)\succ 0$ and $W_d(X,\M)\succ 0$ are satisfied by the same set of certificates, since 
$\widehat W_d(X,\M)$ is congruent to the block diagonal matrix $\diag{X,W_d(X,\M)}$.
We can thus consider the expanded LMI
\begin{equation*}
\widehat W_d(X,\M+\Delta_\M):= \left[\begin{array}{ccc} X &  X(A + \Delta_A) & X(B+\Delta_B) \\  (A^{\mathsf{H}}+\Delta_A^{\mathsf{H}})X & E^{\mathsf{H}}XE & C^{\mathsf{H}}+\Delta_C^{\mathsf{H}}  \\ (B^{\mathsf{H}}+\Delta_B^{\mathsf{H}})X &
C+\Delta_C & D^{\mathsf{H}}+\Delta_D^{\mathsf{H}} + D+\Delta_D  \end{array}\right] \succ 0,
\end{equation*}
which is now an LMI in the unknown parameters of $\Delta_\M$, for fixed 
matrices $X$ and $E$. Setting
\begin{equation*} %\label{defWhatX}
\widehat W_0:=\left[\begin{array}{ccc} X &  XA & XB \\  A^{\mathsf{H}}X & E^{\mathsf{H}}XE & C^{\mathsf{H}} \\ B^{\mathsf{H}}X &
C & D^{\mathsf{H}} + D  \end{array}\right], \quad  \left[\begin{array}{c} \widehat F \\  \widehat G \end{array}\right]:= 
\left[\begin{array}{ccc} X & 0 & 0 \\ 0 & 0 & I_m\\ \hline 0 & I_n & 0 \\ 0 & 0 & I_m \end{array}\right], 
\end{equation*}
and using the matrix $\Delta$ in \eqref{Delta}, this inequality can be written as the structured LMI
\begin{equation} \label{WDeltahat}
\widehat W_0 + \left[\begin{array}{cc} \widehat F^\mathsf{H} & \widehat G^\mathsf{H} \end{array}\right] \left[\begin{array}{cc} 0 & \Delta \\ \Delta^\mathsf{H} & 0 
\end{array}\right] \left[\begin{array}{c} \widehat F \\ \widehat G\end{array}\right] \succ 0,
\end{equation}
as long as the system is still strictly passive. In order to violate this condition, we need to find the smallest $\Delta$ such that the determinant of $\widehat W_d(X,\M+\Delta_\M)$ becomes 0. 
Clearly, this is now very similar to the CT case.

\subsection[The X-passivity radius]{The $X$-passivity radius}

We now define the so-called \emph{$X$-passivity radius}, which was introduced in \cite{BeaMV19} and gives a bound for the usual passivity radius. We will use in this subsection the notation $W(X,\M)$ for both the CT
matrix $W_c(X,\M)$ and the DT matrix $\widehat W_d(X,\M)$.
\begin{definition}
For $X\in \XWpdpd$ the  \emph{$X$-passivity radius} is defined as
\[
	\rho_\M(X):= \inf_{\Delta_\M\in \mathbb C^{n+m,n+m}}\left\{ \| \Delta_\M \| \; | \; \det W(X,\M+\Delta_\M) = 0\right\}.
\]
\end{definition}
Note that in order to compute $\rho_\M(X)$ for the model $\M $, we must have a matrix $X\in \XWpdpd$, since $W(X,\M)$  must be positive definite to start with and also $X$ should be positive definite to obtain a state-space transformation from it. The following relation between the $X$-passivity radius and the usual passivity radius was also given in \cite{BeaMV19}.
\begin{lemma} \label{lem:pasrad}
	The passivity radius for a given model $\M$ satisfies
	\begin{equation}
\nonumber
	\rho_{\M}:= \sup_{X\in \XWpdpd}\inf_{\Delta_\M\in \mathbb C^{n+m,n+m}}\{\| \Delta_\M \| | \det W(X,\M+\Delta_\M)=0\}= \sup_{X\in \XWpdpd} \rho_{\M}(X).\label{passive}
	\end{equation}
\end{lemma}
Notice that any boundary point of $\XWpd$ which is not in $\XWpdpd$ has an $X$-passivity radius equal to 0, since $\det W(X,\M)=0$.
We now show that the $X$-passivity radius can be reduced to a one dimensional optimization problem. 
Using the Cholesky decomposition $W_0 := R^\mathsf{H}R$, the determinant of $W(X,\M+\Delta_\M)$ is also equal to
\begin{equation*} %\label{eq:2block}
 \det(W_0)\det\left( I_n + 
\left[ \begin{array}{cc} \gamma R^\mathsf{-H}F^\mathsf{H} & R^\mathsf{-H}G^\mathsf{H}/ \gamma \end{array}\right]
\left[ \begin{array}{cc} 0 & \Delta \\ \Delta^\mathsf{H} & 0 \end{array}\right]
\left[ \begin{array}{cc} \gamma F R^{-1} \\ G R^{-1}/  \gamma\end{array}\right]\right).
\end{equation*}
The following theorem, based on results of \cite{OveV05} and \cite{BeaMV19}, gives the minimal perturbation $\Delta$ in both the Frobenius norm and the spectral norm.  
\begin{theorem} \label{thm:mingamma}
	Consider the matrices $F, G, W_0$ as in~\eqref{WDelta} or~\eqref{WDeltahat} and the pointwise positive semi-definite matrix function
\begin{equation}\label{defmgamma}
M(\gamma):=\left[ \begin{array}{cc} \gamma F \\ G / \gamma\end{array}\right] W_0^{-1}
\left[ \begin{array}{cc} \gamma  F^\mathsf{H} & G^\mathsf{H} / \gamma \end{array}\right]
\end{equation}
in the real parameter $\gamma$. Then the largest eigenvalue $\lambda_{\max}(M(\gamma))$ is a \emph{unimodal function} of $\gamma$, ({i.e.} it is first monotonically decreasing and then monotonically increasing with growing $\gamma$). At the minimizing value $\underline \gamma$,  $M(\underline{\gamma})$ has an eigenvector $z$, {i.e.}
\[
 M(\underline{\gamma}) z = \underline\lambda_{\max} z, \quad z:=\left[ \begin{array}{cc} u \\ v \end{array}\right],
 \]
where
$  \|u\|_2^2=\|v\|_2^2=1$.
The minimum norm perturbation $\Delta$ is of rank $1$ and is given by $\Delta=vu^{\mathsf{H}}/\underline{\lambda}_{\max}$. It has norm $1/\underline{\lambda}_{\max}$
both in the spectral norm and in the Frobenius norm.
\end{theorem}
It thus follows that we want to minimize $\underline{\lambda}_{\max}$.
In \cite{BeaMV19} also the following  simple bound for $\underline{\lambda}_{\max}$ was derived.
\begin{corollary}\label{cor:lev} Consider the matrices $F, G, W_0$ in \eqref{WDelta}\eqref{WDeltahat} and the pointwise positive semi-definite matrix function $M(\gamma)$ as in (\ref{defmgamma}). The largest eigenvalue of $M(\gamma)$ is also the largest eigenvalue of
\begin{equation} \label{maxeig}
\gamma^2 R^\mathsf{-H} F^\mathsf{H}F R^{-1} + 
R^\mathsf{-H} G^\mathsf{H}G R^{-1}/\gamma^2.
\end{equation}
An upper bound for $\underline{\lambda}_{\max}$ is given by $\underline{\lambda}_{\max}\le \frac{2}{\alpha\beta}$ where $\alpha^{-2}:=\lambda_{\max}(F W_0^{-1} F^\mathsf{H})$ and $\beta^{-2}=\lambda_{\max}(G W_0^{-1} G^\mathsf{H})$. The corresponding lower bound for $\| \Delta \|_2$ and $\| \Delta \|_F$ is given by
\[
  \rho_\M(X)  \ge \alpha\beta/2.
\]
\end{corollary}
The result stated in this corollary is based on the following theorem, also proven in \cite{BeaMV19}.
\begin{theorem}\label{thm:Xpassivity}
Let $\M=\{A,B,C,D,E\}$ be a given model and assume that we are given a matrix $X\in \XWpdpd$, then the $X$-passivity radius $\rho_\M(X)$ is bounded by
\[
 \alpha\beta/2 \le \rho_\M(X) \le  \alpha\beta/(1+|v^{\mathsf{H}}w|),
 \]
where $v$, $\hat v$, $w$ and $\hat w$ are vectors of norm $1$, satisfying
\[\alpha^{-2}=\lambda_{\max}(F W_0^{-1} F^\mathsf{H}), \;\beta^{-2}=\lambda_{\max}(GW_0^{-1} G^\mathsf{H}), \; R^\mathsf{-H}F^\mathsf{H} \hat v=v/\alpha, \;  R^\mathsf{-H}G^\mathsf{H}\hat w=w/\beta.
\]
Moreover if $v$ and $w$ are linearly dependent, then $\rho_\M(X)=\alpha\beta/2$.
\end{theorem} 
These results can of course be applied to normalized systems as well. More importantly, 
we can use them to compare the $X$-passivity radius of the original model $\M$
and the transformed model $\M_T$, for both the CT and DT cases. 
\begin{theorem} \label{th:optimal}
The normalized $X$-passivity radius $\rho_{\M_T}(c^2I_n)$ is larger or equal to the 
corresponding $X$-passivity radius $\rho_\M(X)$, for both the CT and DT cases, provided we choose $c^2:=\lambda_{\min}(X)$ as normalization constant. 
\end{theorem}
\begin{proof}
We only give the proof for the CT case since both proofs are essentially the same. Notice that $c^2:=\lambda_{\min}(X)$ implies $\sigma_{\min}(T)=1$ and hence $T^\mathsf{H}T\succeq I_n$. 
We then start from the formulations of the passivity radii given in \eqref{maxeig}~:
\[  \rho_{\M}^{-1}(X)=\min_{\gamma}\| \gamma^2 R^\mathsf{-H}F^\mathsf{H}FR^{-1}  + R^\mathsf{-H}G^\mathsf{H}GR^{-1}/ \gamma^2 \|^2,
\]
and
\[  \rho_{\M_T}^{-1}(c^2I_n)=\min_{\gamma}\| \gamma^2 R_T^\mathsf{-H}F_T^\mathsf{H}F_TR_T^{-1}  + R_T^\mathsf{-H}G_T^\mathsf{H}G_TR_T^{-1}/ \gamma^2 \|^2.
\]
It follows from \eqref{gssPH} and \eqref{defWX} that 
\[ \hat T = \left[\begin{array}{cc} T & 0 \\ 0 & I_m \end{array}\right], \quad R_T^\mathsf{-H}F_T^\mathsf{H}=R^\mathsf{-H}F^\mathsf{H}\hat T^{-1}, \quad   R_T^\mathsf{-H}G_T^\mathsf{H}=R^\mathsf{-H}G^\mathsf{H}\hat T^{-1},\]
which implies that
\[  \rho_{\M_T}^{-1}(c^2I_n)=\min_{\gamma}\| \gamma^2 R^\mathsf{-H}F^\mathsf{H}\hat T^{-1} \hat T^\mathsf{-H}FR^{-1}  + R^\mathsf{-H}G^\mathsf{H}\hat T^{-1} \hat T^\mathsf{-H}GR^{-1}/ \gamma^2 \|^2,
\]
and since $\hat T^\mathsf{H}\hat T \succeq I_{n+m}$ the right hand side is non-increasing for all $\gamma$, including the 
minimizing $\gamma$. Therefore $\rho_{\M_T}(c^2I_n) \ge \rho_\M(X)$.
\end{proof}

Moreover, in the CT case, the transformed model is in port-Hamiltonian form~:
$$
 {\M}_T:= \{A_T,B_T,C_T,D_T,E_T\}:=\{c^2(J-R)E_T,G -K,c^2(G^{\mathsf{H}} + K^{\mathsf{H}})E_T,S+N,E_T\}.
$$ 

\begin{remark}
Since the certificates are always positive definite, it follows that $c>0$. 
But it is possible that the coefficient $c$ is quite small. In such a case, we recommend to increase $c$ until reaching $\rho_{\M_T}(c^2I_n)=\rho_\M(X)$. 
\end{remark}

\section[Extremal values xi-star of (X) for a certificate X of a given model M]{Extremal values $\xi^*(X)$ for a certificate $X$ of a given model $\M$} \label{sec:maxpass}

Here we link the passivity radius $\rho_\M$ of a given model to its $X$-passivity radius $\rho_\M(X)$ via an extremal value $\xi^*(X)$ of a function $\xi(X)$ linked to a particular LMI. Since the formulas are quite different for the CT and DT cases, we treat them in separate subsections. 

\subsection{The CT case}

It follows from the previous section that the interplay between $W_c(X,\M)$, $E$ and $X$ is important for estimating the passivity radius $\rho_\M$ of a given model $\M$. We will characterize this using the extremal value $\xi^*(X)$ of a given certificate, associated with the following constrained LMI
\begin{equation} \label{xi}
W_c(X,\M) \succeq \xi \diag{E^{\mathsf{H}}XE,I_m},
\end{equation}
where $\xi$ is a real parameter, $\diag{A,B}$ denotes a block diagonal matrix with diagonal blocks $A$ and $B$, and obtain the following theorem.

\begin{theorem}  \label{thm:maxoverX}
Let $\M:=\{A,B,C,D,E\}$ be a minimal realization of a passive system with non-singular $E$, and let $X$ be a fixed certificate in the corresponding set $\XWpd$. Then there is a unique $\xi^*(X)\ge 0$ which is maximal for the matrix inequality $$W_c(X,\M) \succeq \xi \diag{E^{\mathsf{H}}XE,I_m},$$ 
and is equal to the smallest eigenvalue of the positive definite pencil 
$W_c(X,\M) -\xi \diag{E^{\mathsf{H}}XE,I_m}$. This maximal value $\xi^*(X)$ equals zero if $X$ is a boundary point of $\XWpd$ and is positive if and only if $X$ is in $\XWpdpd$.  Moreover, $X$ is a valid certificate for the passivity of the shifted system 
\begin{equation} \label{eq:shiftedCT}
\M_\xi=\{A+\frac{\xi}{2}E,B,C,D-\frac{\xi}{2}I_m,E\}.
\end{equation}
\end{theorem}
\begin{proof}
If $X$ is a boundary point of $\XWpd$ then $\det W_c(X,\M)=0$ and for those $X$, we have $\xi^*(X)=0$.
If $X$ belongs to $\XWpdpd$, then $W_c(X,\M)\succ 0$ and $\diag{E^{\mathsf{H}}XE,I_m}\succ 0$.
Therefore the pencil $W_c(X,\M) -\xi \diag{E^{\mathsf{H}}XE,I_m}$ is positive definite and all its eigenvalues are then positive. Its smallest eigenvalue is clearly equal to $\xi^*(X)$ and hence $\xi^*(X)>0$. Conversely, if $\xi^*(X)>0$ then $W_c(X,\M)\succ 0$ and $X\in \XWpdpd$. 
Next, we have that for a given $X$ in $\XWpdpd$, the passivity LMI $W_c(X,\M_\xi)$ for the modified model $\M_\xi :=\{A+\frac{\xi}{2} E,B,C,D-\frac{\xi}{2} I_m,E\}$ is given by
\begin{equation*} %\label{shiftedc}
W_c(X,\M_\xi) := \! \left[\begin{array}{cc} \! -(A^\mathsf{H}+\frac{\xi}{2} E^{\mathsf{H}})XE-E^{\mathsf{H}}X(A+\frac{\xi}{2} E) &   C^{\mathsf{H}}-XB \\    C-B^{\mathsf{H}}X & D^\mathsf{H}+ D-\xi I_m \end{array}\right] \succeq 0.
\end{equation*}
Clearly this inequality follows from \eqref{xi} for $0 \le \xi \le \xi ^*(X)$ and its determinant is zero if $ \xi =\xi ^*(X)$. \hfill
\end{proof}

We also have the following Lemma.
\begin{lemma} \label{inclusionc}
For every $X\succ 0$ in $\XWpdpd$ and any $\xi_-$ and $\xi_+$ satisfying $0\le \xi_- < \xi_+ \le \xi^*(X)$, the  systems $\M_{\xi_-}$ and $\M_{\xi_+}$ are passive. Moreover, the whole solution set of
$W_c(X,\M_{\xi_+})\succeq 0$ is included in the solution set of  $W_c(X,\M_{\xi_-})\succ 0$.
\end{lemma}
\begin{proof}
The LMIs for two different values $\xi_-<\xi_+$ are related as
\[
 W_c(X,\M_{\xi_+}) = W_c(X,\M_{\xi_-}) - (\xi_+-\xi_-) \diag{E^{\mathsf{H}}XE,I_m}.
\]
Since $X\in\XWpdpd$, $\xi^*(X)>0$ and $\diag{E^{\mathsf{H}}XE,I_m}\succ 0$, it follows that
\begin{equation} \label{ineqsc}
W_c(X,\M) \succeq W_c(X,\M_{\xi_-}) \succ W_c(X,\M_{\xi_+})\succeq W_c(X,\M_{\xi^*(X)}) \succeq 0. \hfill
\end{equation}
Hence, the systems $\M_{\xi_-}$ and $\M_{\xi_+}$ are passive, since the associated LMIs have a nonempty solution set. Now consider {\em any} $X$ for which $W_c(X,\M_{\xi_+})\succ 0$.
Since $\xi_+$ is positive, so is $\xi^*(X)$ and hence $X\in \XWpdpd$. It then follows from \eqref{ineqsc} that $W_c(X,\xi_-)\succ 0$. Hence,
the solution set of $W_c(X,\M_{\xi_+}) \succeq 0$ is included in the solution set of $W_c(X,\M_{\xi_-}) \succ 0$.
\end{proof}

Lemma~\ref{inclusionc} implies that for a given $X\in \XWpdpd$, the solution sets of  $W_c(X,\M_\xi) \succ 0$ are shrinking with increasing $\xi$.
\begin{theorem} \label{th:pasrad}
Let $\M$ be a given passive model with transfer function ${\cal T}(s)$ and let $X=c^2T^\mathsf{H}T$ be a certificate for the corresponding LMI $W_c(X,\M)\succeq 0$ with invertible $E$, then the extremal value $\xi^*(X)$ of the LMI $W_c(X,\M) - \xi \diag{X,I_m}\succ 0$ is the smallest eigenvalue of the positive definite pencil
\begin{equation} \label{eq:posdef}
    W_c(X,\M)-\xi \diag{E^\mathsf{H}XE,I_m}. 
\end{equation} 
Since $E$ is invertible, this is also the smallest eigenvalue of the normalized matrix
$$ \widetilde W(c^2I,\widetilde \M_T):=\left[\begin{array}{cc} -c^2(A_TE_T^{-1})^\mathsf{H} - c^2 (A_TE_T^{-1}) & (C_TE_T^{-1})^\mathsf{H}-c^2B_T \\ C_TE_T^{-1} - c^2B_T^\mathsf{H} & D_T+D_T^\mathsf{H}\end{array}\right],
$$
corresponding to the model $\widetilde \M:=\{A_TE_T^{-1},B_T,C_TE_T^{-1},D_T,I_n\}$ with transfer function ${\cal T}(s)$.
\end{theorem}
\begin{proof}
The matrices $W_c(X,\M)$ and $\xi \diag{E^\mathsf{H}XE,I_m}$ are positive definite, which implies that all eigenvalues are positive, and are equal to the eigenvalues of 
$$ \widetilde W_c(c^2I,\widetilde \M_T)= \diag{(TE)^{-\mathsf{H}},I_m}
W_c(X,\M)\diag{(TE)^{-1},I_m}.
$$  
The smallest value of $\xi$ for which the pencil \eqref{eq:posdef} becomes singular is therefore the smallest eigenvalue of that matrix. 
\end{proof}

The extremal value $\xi^*(X)$ is a function of the chosen $X$ and we would like to know which value $\xi^*(X)$ corresponds to the largest possible passivity radius. We now show that we can obtain this value by relating it to the passivity of the transfer function
\[
  \T_\xi(s):=C((s-\xi/2)E -A)^{-1}B+(D-\xi I_m/2),
\]
of the modified system $\M_\xi$, and this without finding $X$ first. Note that we have  assumed that the associated system is minimal, a property which is not changed by the shift.
It follows from the discussion of Section \ref{sec:PH} that this transfer function corresponds to a {\em strictly} passive system if and only if (i) the transfer function  $\T_\xi(s)$ is
asymptotically stable and (ii) the matrix function $\Phi_\xi(s):=\T^\mathsf{H}_\xi(-s)+\T_\xi(s)$ is positive definite on the $\imath \omega$ axis, with $\omega=\pm\infty$ included.
It has been presented in Section \ref{sec:PH} that the zeros of $\Phi_\xi(s)$ are also the eigenvalues of the Hamiltonian matrix
$H_\xi$ given by
% {\small
\begin{equation*}
\left[ \begin{array}{cc} AE^{-1}+\frac{\xi}{2}I_n & 0 \\ 0 & -E^{-\mathsf{H}}A^{\mathsf{H}}-\frac{\xi}{2} I_n \end{array}\right] + \left[ \begin{array}{cc}  B \\ E^{-\mathsf{H}}C^{\mathsf{H}} \end{array}
\right] (D^{\mathsf{H}}+D-\xi I_m)^{-1} \left[ \begin{array}{cc} CE^{-1} & -B^{\mathsf{H}} \end{array} \right],
\end{equation*}
% }
%
provided that $D^{\mathsf{H}}+D-\xi I_m \succ 0$ and the realization of $\M_\xi$ is minimal.
The three algebraic conditions corresponding to strict passivity of $\T_\xi(s)$ are, therefore, given by
\begin{enumerate}
	\item [A1.] $A+\frac{\xi}{2} E- s E $ has all its eigenvalues in the open left half plane (stability).
	\item [A2.] $D^{\mathsf{H}}+D-\xi I_m$ has positive eigenvalues (positive-realness at $\omega=\infty$).
	\item [A3.] $H_\xi$ has no eigenvalues on the $\imath\omega$ axis  (positive-realness at finite $\omega$).
\end{enumerate}
All of these conditions are phrased in terms of eigenvalues of certain matrices that depend on the parameter $\xi$. Since eigenvalues are continuous functions of the matrix elements,
one can consider limiting cases for the above conditions.
As explained in Section \ref{sec:PH}, the passive transfer functions are limiting cases of the strictly passive ones. These limiting cases correspond to the
first value of $\xi$ where one of the three algebraic conditions fails. Note that condition A3. is more robustly expressed in terms of the eigenvalues of the matrix pencil
\begin{equation*} %\label{statespaceXi}
S_\xi(s) :=
\left[ \begin{array}{cc|c} 0 & A+\frac{\xi}{2}E-sE & B \\
	A^{\mathsf{H}}+\frac{\xi}{2}E+sE & 0 & C^{\mathsf{H}} \\ \hline B^{\mathsf{H}} & C & D^{\mathsf{H}}+D-\xi I_m  \end{array} \right].
\end{equation*}

It is obvious  that the conditions A1.-A3. are not satisfied anymore for large enough $\xi\in \mathbb{R}$. For instance, for $\xi > \lambda_{\min}(D^{\mathsf{H}}+D)$ the second condition fails and $\lambda_{\min}(D^{\mathsf{H}}+D)$ is thus a simple upper bound for $\xi^*(X)$ for any $X$.

In Subsection \ref{ss:HEC}, we will show that finding 
the supremum over all values $\xi^*(X)$ can be solved using the Hybrid Expansion Contraction method described in \cite{MitV23}.

\subsection{The DT Case}

The discussion  of this case is very similar to the CT case, but the LMI's and equations are quite different. We consider the constrained LMI
\begin{equation} \label{xid}  
 \widehat W_d(X,\M) \succeq \xi \diag{X,E^{\mathsf{H}}XE,I_m}.
\end{equation}
Then the following theorem gives a bound on how large we can choose $\xi$ in this LMI.
\begin{theorem}  \label{thm:maxoverXd}
Let $\M:=\{A,B,C,D,E\}$ be a minimal realization of a passive system with non-singular $E$, and let $X$ be a fixed certificate in the corresponding set $\XWpd$. Then there is a unique $\xi^*(X)$ which is maximal for the matrix inequality \eqref{xid}, is strictly smaller than 1 and is equal to the smallest eigenvalue of the positive definite pencil 
$ \widehat W_d(X,\M) - \xi \diag{X,E^{\mathsf{H}}XE,I_m}$. This maximal value $\xi^*(X)$ equals zero if $X$ is a boundary point of $\XWpd$ and is positive if and only if $X$ is in $\XWpdpd$.  Moreover, $X$ is a valid certificate for the passivity of the shifted system \begin{equation} \label{eq:shiftedDT}
\M_\xi=\{A_\xi,B_\xi,C_\xi,D_\xi,E_\xi\}
:=\left\{\frac{A}{(1-\xi)},\frac{B}{(1-\xi)},\frac{C}{(1-\xi)},\frac{D-\xi I_m/2}{(1-\xi)},E\right\}.
\end{equation}
\end{theorem}
\begin{proof} 
The proof is completely analogous to the proof of Theorem \ref{thm:maxoverX} in the CT case, except for the shifted model \eqref{eq:shiftedDT}, which is obtained as follows.
The equality in the left hand side of the LMI
\begin{equation*} %\label{shiftedd}  
(1-\xi)\widehat W_d(X,\M_\xi) = \widehat W_d(X,\M) - \xi \left[\begin{array}{ccc} X  & 0 & 0 \\ 0 & E^{\mathsf{H}}XE & 0 \\ 0 & 0 & I_m \end{array}\right]\succeq 0
\end{equation*} 
is clearly satisfied by the matrix
$$
\widehat W_d(X,\M_\xi):= \left[\begin{array}{ccc} X  &  XA_\xi & XB_\xi \\ 
A_\xi^{\mathsf{H}}X & E^{\mathsf{H}}XE & C_\xi^{\mathsf{H}} \\   B_\xi^{\mathsf{H}}X & C_\xi & D_\xi^{\mathsf{H}}+D_\xi \end{array}\right]
$$ 
which corresponds to the modified model \eqref{eq:shiftedDT}, and also implies that $\xi$ must be smaller than 1.
\end{proof}

The following Lemma is again analogous to the CT case and is therefore given without proof.
\begin{lemma} \label{inclusiond}
For every $X\succ 0$ in $\XWpdpd$ and any  $\xi_-$ and $\xi_+$ satisfying  $0\le \xi_- < \xi_+  \le \xi^*(X)$, the passivity LMIs for the systems $\M_{\xi_-}$ and $\M_{\xi_+}$ are satisfied. Moreover, the whole solution set of 
$\widehat W_d(X,\M_{\xi_+})\succeq 0$ is included in the solution set of  $\widehat W_d(X,\M_{\xi_-})\succ 0$.
\end{lemma}
Lemma~\ref{inclusiond} implies that for a given $X\in \XWpdpd$, the solution sets of  $\widehat W_d(X,\M_\xi) \succ 0$ are shrinking with increasing $\xi$.
\begin{theorem} \label{th:pasradd}
Let $\M$ be a given passive model with transfer function ${\cal T}(z)$ and let $X=c^2T^\mathsf{H}T$ be a certificate for the corresponding LMI $\widehat W_d(X,\M)\succeq 0$ with invertible $E$, then the extremal value $\xi^*(X)$ of the LMI $\widehat W_d(X,\M) - \xi \diag{X,E^\mathsf{H}XE,I_m}$ is the smallest eigenvalue of the positive definite pencil
\begin{equation} \label{eq:posdefd}
    \widehat W_d(X,\M)-\xi \diag{X,E^\mathsf{H}XE,I_m}. 
\end{equation} 
Since $E$ is invertible, this is also the smallest eigenvalue of the normalized matrix $$\widehat W_d(c^2I,\widetilde \M_T):=
\left[\begin{array}{ccc}
c^2 I_n & c^2(A_TE_T^{-1}) & c^2 B_T \\ c^2 (A_TE_T^{-1})^\mathsf{H}  & c^2 I_n & (C_TE_T^{-1})^\mathsf{H} \\ c^2B_T^\mathsf{H} & C_TE_T^{-1}  & D_T+D_T^\mathsf{H}\end{array}\right],
$$
corresponding to the model $\widetilde \M:=
\{A_TE_T^{-1},B_T,C_TE_T^{-1},D_T,I_n\}$ with transfer function ${\cal T}(z)$.
\end{theorem}
\begin{proof}
The matrices $\widehat W_d(X,\M)$ and $\diag{X,E^\mathsf{H}XE,I_m}$ are positive definite, which implies that all eigenvalues are positive, and are equal to the eigenvalues of 
$$ \widehat W_d(c^2I,\widetilde \M_T)= 
\diag{T^{-\mathsf{H}},(TE)^{-\mathsf{H}},I_m}
W_c(X,\M)\diag{T^{-1},(TE)^{-1},I_m}.
$$  
The smallest positive value of $\xi$ for which the pencil \eqref{eq:posdefd} is singular is therefore the smallest eigenvalue of that matrix. 
\end{proof}

Also here, we can find the extremal value $\xi^*(X)$ by relating it to the passivity of the transfer function of the modified system $\M_\xi$ which is minimal since $\M$ was assumed to be minimal~:
$$  \T_\xi(z):=C_\xi(zE -A_\xi)^{-1}B_\xi+D_\xi,
$$
and this without computing $X$.
It follows from the discussion of Section \ref{sub:DT} that this transfer function corresponds to a {\em strictly} passive system if and only if the following two conditions are satisfied~: 
(i) the transfer function  $\T_\xi(z)$ is asymptotically stable and 
(ii) the matrix function $\Phi_\xi(z):=\T^\mathsf{H}_\xi(z^{-1})+\T_\xi(z)$ is strictly positive on the unit circle $e^{\imath \omega}, \omega\in[-\pi,\pi]$.  
It has been shown in Section \ref{sub:DT} that the zeros of $\Phi_\xi(z)$ are the eigenvalues of the symplectic matrix 
\begin{equation}
S_\xi:=\left[\begin{array}{cc} E &  B_\xi (D_\xi^{\mathsf{H}}+D_\xi)^{-1} B_\xi^{\mathsf{H}} \\
 0 & (A_\xi-B_\xi (D_\xi^{\mathsf{H}}+D_\xi)^{-1} C_\xi)^{\mathsf{H}} \end{array}\right]^{-1} \left[\begin{array}{cc} A_\xi-B_\xi (D_\xi^{\mathsf{H}}+D_\xi)^{-1} C_\xi & 0 \\
C_\xi^{\mathsf{H}} (D_\xi^{\mathsf{H}}+D_\xi)^{-1} C_\xi & E \end{array}\right],
\end{equation}
which are also the finite eigenvalues of the pencil
\begin{equation*}  
z  \left[ \begin{array}{ccc}  0 & -E & 0 \\ A_\xi^\mathsf{H} & 0 & 0 \\ 
B_\xi^{\mathsf{H}} & 0 & 0 \end{array}
\right] + \left[ \begin{array}{ccc} 0 & A_\xi & B_\xi \\ -E^{\mathsf{H}} & 0 & C^{\mathsf{H}}_\xi \\
0 & C_\xi & D^{\mathsf{H}}_\xi+D_\xi \end{array}\right] 
\end{equation*}
or equivalently, those of the pencil
\begin{equation} \label{xipencil}  
z  \left[ \begin{array}{ccc}  0 & (\xi-1) E & 0 \\ A^\mathsf{H} & 0 & 0 \\ 
B^{\mathsf{H}} & 0 & 0 \end{array}
\right] + \left[ \begin{array}{ccc} 0 & A & B \\ (\xi-1)E^{\mathsf{H}} & 0 & C^{\mathsf{H}} \\
0 & C & D^{\mathsf{H}}+D - \xi I_m \end{array}\right] 
\end{equation}
and that the realization of $\M_\xi$ is minimal. 
The algebraic conditions corresponding to strict passivity of $\T_\xi(z)$ are therefore
\begin{enumerate}
	\item[A1.] $A_\xi $ has all its eigenvalues inside the unit disc (stability)
	\item[A2.] the pencil \eqref{xipencil} has no eigenvalues on the unit circle  (positive realness).
\end{enumerate}
These conditions are phrased in terms of eigenvalues of certain matrices that depend on the parameter $\xi$. Since eigenvalues are continuous functions of the matrix elements, one can consider limiting cases for the above conditions. As explained in Section \ref{sub:DT} the passive transfer functions are limiting cases of strictly passive ones. Those limiting cases correspond to the value of $\xi$ where one of the above conditions does not hold anymore.

\subsection{Root-min problems}  \label{ss:HEC}

The following theorem, proven in \cite{MehV19} and \cite{MehV20},
shows that the best certificate for a given model, is the supremum of all functions $\xi^*(x)$ for all $X\in \XWpd$. We therefore leave out the proof.

\begin{theorem} \label{thm:Xid} Let $\M$ be a strictly passive and minimal system. Then there is a bounded supremum $\Xi:=\sup_\xi \{\xi \; | \; \T_\xi(\lambda) \mathrm{\; is \; strictly \; passive}\}$ where the following properties hold
\begin{enumerate}
\item  $\T_\Xi(\lambda)$ is passive,
\item the solution set of $W(X,\M_\Xi)\succeq 0$ is not empty,
\item the solution set of $W(X,\M_\Xi)\succ 0$ is empty,
\item for any $\xi < \Xi$ the solution set of $W(X,\M_\xi)\succ 0$ is non-empty,
\item  $\Xi:=\sup_X \xi^*(X)$ for all $X\in \XWpd$.
\end{enumerate}
\end{theorem}

It is shown in \cite{MitV23} that the passivity radius of a CT or DT system model $\M$ is a root-min problem, which can be solved efficiently using the so-called Hybrid Expansion-Contraction (HEC) algorithm and has ultimate quadratic convergence (we refer to \cite{MitV23} for further details). The following theorem was also shown in \cite{MehV19} and \cite{MehV20} and states important properties that the HEC 
algorithm requires to efficiently compute the supremal value $\Xi$ of a given model $\M$ of a passive transfer function ${\cal T}(\lambda)$.
The theorem applies as well to the CT case (where $W(X,\M_\xi)$ stands for $W_c(X,\M_\xi)$) as to the DT case (where $W(X,\M_\xi)$ stands for $\widehat W_d(X,\M_\xi)$). 

Notice that $\Xi$ is associated with a given model $\M$ but does not require the calculation of any certificate $X$. If one wants to reconstruct a certificate $X^*$ that satisfies $W(X^*,\M_\Xi)\succeq 0$ then one can  solve the stabilizing Riccati solution $X_-$ of the shifted model $\M_\Xi$.

\subsection{Improving robustness of a model}

We point out that the value of $\Xi$ is independent of the model used for representing the transfer function ${\cal T}(\lambda)$. This follows from the fact that the transformations $S$ and $T$ described in \eqref{gsst} commute with the shifting transformations described in \eqref{eq:shiftedCT} and \eqref{eq:shiftedDT}. It then follows from 
Theorem \ref{th:optimal} that one can find optimal models in the set of 
normalized models. This was already observed in \cite{MehV19} where 
port-Hamiltonian models in state-space form were shown to contain optimally
robust models but no construction was given explicitly. 

\section{Concluding remarks} \label{sec:conclude}

In this paper we have provided a definition of the passivity radius of a generalized state-space system of a rational matrix and also showed how to optimize it,
provided a particular normalization constant $c$ is chosen appropriately.
Constructing a generalized state-space realization with maximal passivity radius can be used to improve the robustness of that model, against perturbations of the model parameters.

\end{document}